\documentclass{amsart}
\usepackage{amsmath,amssymb}
\usepackage{color,mathdots}

\newtheorem{theorem}{Theorem}[section]
\newtheorem{lemma}[theorem]{Lemma}
\newtheorem{corollary}[theorem]{Corollary}
\newtheorem{fact}[theorem]{Fact}

\theoremstyle{definition}

\newtheorem{remark}[theorem]{Remark}
\newtheorem{definition}[theorem]{Definition}

\def \d {\delta}

\def \D {\Delta}
\def \DCF{\operatorname{DCF}}
\def \dcl{\operatorname{dcl}}

\def \tp{\operatorname{\tp}}

\def \U {\mathbb U}
\def \M {\mathcal M}

\def\Ind#1#2{#1\setbox0=\hbox{$#1x$}\kern\wd0\hbox to 0pt{\hss$#1\mid$\hss}
\lower.9\ht0\hbox to 0pt{\hss$#1\smile$\hss}\kern\wd0}
\def\ind{\mathop{\mathpalette\Ind{}}}
\def\Notind#1#2{#1\setbox0=\hbox{$#1x$}\kern\wd0\hbox to 0pt{\mathchardef
\nn=12854\hss$#1\nn$\kern1.4\wd0\hss}\hbox to
0pt{\hss$#1\mid$\hss}\lower.9\ht0 \hbox to
0pt{\hss$#1\smile$\hss}\kern\wd0}
\def\nind{\mathop{\mathpalette\Notind{}}}

\title{Zilber dichotomy for $\DCF_{0,m}$}

\author{Omar Le\'on S\'anchez}
\address{Omar Le\'on S\'anchez, Department of Mathematics, University of Manchester, Oxford Road, Manchester, United Kingdom M13 9PL}
\email{omar.sanchez@manchester.ac.uk}

\date{\today}
\thanks{{\em Acknowledgements}: The author was partially supported by EPSRC grant EP/V03619X/1}

\begin{document}

\maketitle

\begin{abstract}
We prove that the theory of differentially closed fields of characteristic zero in $m\geq 1$ commuting derivations DCF$_{0,m}$ satisfies the expected form of the dichotomy. Namely, any minimal type is either locally modular or nonorthogonal to the (algebraically closed) field of constants. This dichotomy is well known for finite-dimensional types; however, a proof that includes the possible case of infinite dimension does not explicitly appear elsewhere.
\end{abstract}

\tableofcontents

\section{Introduction}

Generally speaking, Zilber's dichotomy states that any strongly minimal structure with nonlocally modular geometry is essentially an algebraic curve over an algebraically closed field. While in full generality the dichotomy does not hold \cite{Hru3}, over the years it has become more of a \emph{principle} that draws our attention to a fine structural classification of strongly minimal sets in a particular stable (and more generally simple) theory.  

\medskip

The dichotomy has been shown to hold in the general setup of Zariski geometries~\cite{HruZil}, and this was used in \cite{HruSok} to show that the theory DCF$_{0,1}$ satisfies the dichotomy (by showing that after removing finitely many points, any strongly minimal set is a Zariski geometry). In this case the dichotomy states that any strongly minimal set is either locally modular or nonorthogonal to the field of constants. Furthermore, one can observe that the same holds for any \emph{finite-dimensional} strongly minimal set $X$ in DCF$_{0,m}$ for $m\geq 1$. Indeed, one need only replace ``finite Morley rank" and ``strongly minimal" for ``finite-dimensional'' and ``strongly minimal of finite dimension", respectively, in the statements of \S1 of \cite{HruSok}. Here finite-dimensionality means that for any $a\in X$ the transcendence degree of the differential field generated by $a$ over $K$ is finite, where $K$ is the minimal differential-field of definition of $X$.

\medskip

Some years later in \cite[\S3]{PiZi}, the Canonical Base Property was established for finite-dimensional types in DCF$_{0,1}$ using the machinery of jet spaces. It was noted there that the dichotomy (almost immediately) follows from the CBP. It becomes clear from the arguments there, see in particular \cite[Lemma 3.7(ii)]{PiZi}, that the finite-dimensionality assumption is essential. Nonetheless, the arguments do extend to the partial case and yield the CBP for finite-dimensional types in DCF$_{0,m}$. There are some minor adaptations needed; for instance, the notion of D-variety in the partial case requires an integrability condition due to the commutativity of the derivations. Thus, we take the opportunity to present the proof here in Section \ref{CBP}. Analogous to the case $m=1$, the CBP yields the dichotomy for finite-dimensional types in DCF$_{0,m}$.

\medskip

The finite-dimensional dichotomy, via the CBP, has made an appearance in other contexts of fields with operators such as differential-difference fields DCFA$_0$~\cite{Bustaweak} (or more generally DCFA$_{0,m}$ \cite{Leo2016}) and also for fields equipped with free operators $\mathcal D$-CF$_0$ ~\cite{MooScan} (where operators are \emph{not} required to commute). In these papers, the authors have asked whether the finite-dimensionality assumption could be removed. Note that in~\cite{Bustaexample} it has been noted that in DCFA$_0$ there are in fact strongly minimal sets that are infinite-dimensional, and hence the full dichotomy does not follow from the finite-dimensional case. It was then observed by Bustamante~\cite{Busta}, in the differential-difference context, that the analysis of regular types in DCF$_{0,m}$ of Moosa-Pillay-Scanlon from~\cite{MooPiScan} could be useful to reduce to the finite-dimensional case. In Section~\ref{dichotomy}, we implement Bustamante's idea to prove the dichotomy for arbitrary types in DCF$_{0,m}$ (i.e., not necessarily finite-dimensional).

\medskip

It is worth noting that, while in the differential-difference context DCFA$_0$ \cite{Bustaexample} there are examples of strongly minimal sets that are infinite dimensional, the possible existence of such sets in $\DCF_{0,m}$ for $m\geq 2$ remains an open question (in the case $m=1$ it is known that finite U-rank implies finite-dimensionality~\cite{Poi1}). Of course, if no such examples exist in DCF$_{0,m}$ then the full dichotomy would follow from its finite-dimensional version. It is somewhat surprising that (to my knowledge) we do not know whether the set defined by
$$\d_1(x)=x^3+c, \quad \text{ for $c$ generic,}$$
in $\DCF_{0,2}$ (i.e., two derivations $\d_1$ and $\d_2$) has finite rank or not. Clearly, this set is infinite-dimensional as there is no equation involving $\d_2$. To the author's knowledge there is no definite answer to the aforementioned question and hence, at this point, a proof of the dichotomy without the finite-dimensionality assumption is called for.

\medskip

Throughout we will use the following facts about the theory DCF$_{0,m}$ (see \cite{McGr} for instance): it is a complete $\omega$-stable theory with quantifier elimination (in the language of differential rings) and elimination of imaginaries. Quantifier elimination translates to: the definable sets are Boolean combinations of Kolchin-closed sets. In addition, types are determined by the Kolchin-locus (of a realisation) and the canonical base of the type coincides with the minimal differential-field of definition of its Kolchin-locus.

\


\noindent {\bf Acknowledgements.} The author would like to thank Rahim Moosa for several helpful discussions on the subject.

\

\section{The CBP for finite-dimensional types}\label{CBP}

While the results of this section can be considered as standard adaptations from the ordinary case \cite{PiZi}, there are some subtleties around integrability conditions of $\D$-modules and D-varieties that we wish to spell out. We do, however, keep it brief.

\smallskip

We work in a sufficiently saturated model $(\U,\D)\models \DCF_{0,m}$ and fix a (small) algebraically closed differential subfield $K$. Note that, as a pure field, $\U$ is also a saturated model of ACF; in particular, all algebraic varieties under consideration live in $\U$. We denote the field of $\D$-constants of $K$ and $\U$ by $C_K$ and $C_\U$, respectively. 

\smallskip

Let us briefly recall the notion of jet space of an algebraic variety. Throughout, by an algebraic variety over $K$, we mean an irreducible affine variety defined over $K$. Let  $V$ be such and let $\U[V]$ denote the coordinate ring of $V$ over $\U$ (which is a domain since $K$ is algebraically closed). For $a\in V$, set 
$$\mathcal M_{V,a}:=\{f\in \U[V]:\; f(a)=0\}.$$

\begin{definition}
For $\ell>0$, the $\ell$-jet space of $V$ at $a\in V$, denoted $j_\ell V_a$, is the dual space of the $\U$-vector space $\M_{V,a}/\M^{\ell+1}_{V,a}$.
\end{definition}

For $X$ an algebraic subvariety of $V$ over $K$ and $a\in X$, the containment of $X$ in $V$ yields a $\U$-linear embedding $j_\ell X_a\hookrightarrow j_\ell V_a$ for all $\ell>0$. We identify $j_\ell X_a$ with its image. The following is now a consequence of Nakayama's lemma (see Corollary~2.5 of \cite{PiZi}, for instance).

\begin{fact}\label{unique}
Suppose $X$ and $Y$ are algebraic subvarieties of $V$ over $K$ and $a\in X\cap Y$. If $j_\ell X_a=j_\ell Y_a$ for all $\ell>0$, then $X=Y$.
\end{fact}

While the jet space is defined as an abstract object (the dual of a vector space), it can be identified with a definable set (in fact a subspace of $(\U^d,+)$ for some  $d$) as follows. For $\ell,n>0$, let $\mathcal D_{\ell,n}$ denote the set of differential operators
$$\left\{\frac{\partial ^s}{\partial x_{i_1}^{s_1}\cdots \partial x_{i_r}^{s_r}}:\; 0<s\leq r \text{ and } 1\leq i_1<\cdots < i_r\leq n\right\}.$$ 
If $V\subseteq \U^n$ and $a\in V$, then $j_\ell V_a$ can be identified with the $\U$-linear subspace of $\U^{|\mathcal D_{\ell,n}|}$ defined as
$$\{(u_D)_{D\in \mathcal D_{\ell,n}}\in \U^{|\mathcal D_{\ell,n}|}:\; \sum_{D\in \mathcal D_{\ell,n}}Df(a)\, u_D=0 \text{ for all } f\in I_V\}$$
where $I_V\subseteq K[x_1,\dots,x_n]$ is the ideal of vanishing of $V$. See Lemma 2..3 from \cite{PiZi}.

\medskip

Recall that by a $\D$-module over $(\U,\D)$, it is meant a pair $(E,\mathcal D)$ where $E$ is a finite-dimensional vector space over $\U$ and $\mathcal D=\{D_1,\dots,D_m\}$ with $D_i:E\to E$ additive maps that (pairwise) commute on $V$ such that 
$$D_i(\alpha e)=\d_i(\alpha) \, e \, +\, \alpha \, D_i(e) \text{ for all }\alpha \in \U, \, e\in E.$$
Given a $\D$-module $(E,\mathcal D)$ we define the $\mathcal D$-constants of $E$ as
$$E^\flat=\{e\in E: \, D(e)=0 \text{ for all }D\in \mathcal D\}.$$
Clearly, $E^\flat$ is a $C_{\U}$-vector space (but not necessarily a $\U$-vector space).

\begin{lemma}\label{Cbasis}
Let $(E,\mathcal D)$ be a $\D$-module over $(U,\D)$. Then, there is a $C_{\U}$-basis for $E^\flat$ which is also a $\U$-basis for $E$. 
\end{lemma}
\begin{proof}
This is equivalent to the existence of fundamental systems of solutions to integrable linear differential equations (see Appendix D.1 of \cite{SvdP}). As $\U$ is differentially closed, such fundamental system can be found in $\U$.
\end{proof}

In order to equip jet spaces of finite-dimensional differential-algebraic varieties with a $\D$-module structure, one makes use of the folllowing.

\begin{definition}\label{defdual}
Let $(E,\mathcal D)$ be a $\D$-module over $(\U,\D)$. The canonical $\D$-module structure on the dual $E^*$ is given by the additive operators $\mathcal D^*=\{D_1^*,\dots,D_m^*\}$ defined by 
$$D_i^*(\lambda)(e)=\d_i(\lambda(e))-\lambda(D_i(e)) \quad \text{ for }\lambda\in V^*, e\in E.$$
One can check that this yields a $\D$-module structure on $E^*$; in particular, that the $D_i^*$'s commute with each other (see \cite[Remark 4.5]{Leo2016}).
\end{definition}

Given an (affine) algebraic variety $V\subseteq \U^n$ and $\d\in \D$, the $\d$-prolongation of $V$ is defined as the algebraic variety $\tau_\d V \subseteq \U^{2n}$ with defining equations
$$f(\bar x)=0 \quad  \text{ and } \quad \sum_{i=1}^n \frac{\partial f}{\partial x_i}(\bar x)\cdot y_i+ f^\d(\bar x)=0$$
for $f\in I_V$ (the ideal of vanishing of $V$ over $K$), where $f^\d$ is obtained by applying $\d$ to the coefficients of $f$. In case $V$ is defined over $C_K$, the $\d$-prolongation coincides with the tangent bundle $TV$. More generally, the $\D$-prolongation of $V$, denoted $\tau_\D V\subseteq \U^{n(m+1)}$, is defined as the fibred-product
$$\tau_\D V=\tau_{\d_1} V\times_V \cdots \times_V \tau_{\d_m} V.$$
Note that there is a canonical map $\pi:\tau_\D V\to V$ which in coordinates is given by $\pi(x_0,x_1,\dots,x_m)=x_0$ with each $x_i$ an $n$-tuple. The characteristic property of the $\D$-prolongation is that $(a,\d_1 a, \dots,\d_m a)\in \tau_\D V$ for all $a\in V$. 

By an (affine) algebraic D-variety over $K$ it is meant a pair $(V,s)$ where $V$ is an algebraic variety and $s$ is a regular (algebraic) section of $\pi:\tau_\D V\to V$ with both $V$ and $s$ defined over $K$. In addition, writing the section as $s(\bar x)=(\bar x,s_1(\bar x),\dots, s_m(\bar x))$ with each $s_i=(s_{i,1},\dots,s_{i,n})$ a polynomial map (over $K$), we require the following \emph{integrability} condition
$$\sum_{\ell=1}^n\frac{\partial s_{i,k}}{\partial x_\ell}(a)\cdot s_{j,\ell}(a)+s_{i,k}^{\d_j}(a)=\sum_{\ell=1}^n\frac{\partial s_{j,k}}{\partial x_\ell}(a)\cdot s_{i,\ell}(a) +s_{j,k}^{\d_i}(a)$$
for all $a\in V$, $1\leq i<j\leq m$ and $k=1,\dots,n$. 

\begin{remark}
It is not hard to check that a D-variety structure on $V$ (i.e., the existence of an integrable section $s:V\to \tau_\D V$) is the same as having commuting derivations $\d_1,\dots,\d_m$ on the coordinate ring $\U[V]$ extending the ones on $\U$. Indeed, the unique extensions satisfying $\d_i(x_k)=s_{i,k}(\bar x)$ where $\bar x=(x_1,\dots,x_n)$ are coordinate functions on $V$ yield the desired derivations. The integrability conditions of $s$ translate to these derivations pairwise commuting. See \cite[\S3]{Leo2015} for details and further explanations.
\end{remark}

The set of sharp-points of a D-variety $(V,s=(\operatorname{Id},s_1,\dots,s_m))$ is defined as 
$$(V,s)^\#=\{a\in V:\, (s_1(a),\dots,s_m(a))=(\d_1(a),\dots,\d_m(a))\}.$$
Clearly, if $a\in (V,s)^\#$, then the differential field generated by $a$ over $K$ in $\U$ is just the field $K(a)$.

\smallskip

Collecting Proposition 3.10 and Corollary 3.13 from \cite{Leo2015}, we have the following. 

\begin{fact}\label{Dvar} \
\begin{enumerate}
\item If $(V,s)$ is a D-variety, then $(V,s)^\#$ is Zariski-dense in $V$. Furthermore, any (algebraic-)generic point $a\in V$ over $K$ contained in $(V,s)^\#$ is a differential-generic point of the latter (over $K$).
\item If $a$ is a tuple from $\U$ such that the differential field $K\langle a\rangle$ has finite transcendence degree over $K$, then $K\langle a\rangle$ is the function field $K(V)$ of some D-variety $(V,s)$. Furthermore, up to $\D$-interdefinability, $a$ is a differential-generic point of $(V,s)^\#$ over $K$; in other words, the type $\operatorname{tp}(a/K)$ (in the DCF sense) is determined by 
$$\hspace{1.5cm} \text{``$\, a$ is generic in $V$ over $K$ and $(\d_1(a), \dots, \d_m(a))=(s_1(a),\dots,s_m(a))$"}.$$
\end{enumerate}
\end{fact}

Now, for a D-variety $(V,s)$ and $a\in (V,s)^\#$, the ideal $\M_{V,a}$ of $\U[V]$ is a $\D$-ideal (this is shown to be a $\d$-ideal, for $\d\in \D$, in \cite[Lemma 3.7(iii)]{PiZi}). A posteriori, $\M_{V,a}^\ell$ is also a $\D$-ideal for all $\ell>0$, and hence $\M_{V,a}/\M_{V,a}^{\ell+1}$ inherits the structure of a $\D$-module over $(\U,\D)$. Using Definition~\ref{defdual}, we see that the $\ell$-th jet space $j_\ell V_a$ has a canonical $\D$-module structure.

We now prove the CBP for DCF$_{0,m}$. First, recall that type $p\in S(K)$ is said to be finite-dimensional if $trdeg_K K\langle a \rangle$ is finite for any $a\models p$, where $K\langle a\rangle$ is the differential field generated by $a$ over $K$. Also, a type $q=\operatorname{tp}(d/L)$, over a differential field $L$, is said to be internal to the constants if there is  $b\ind_L d$ and $c$ from $C_\U$ such that $d\in \dcl(L,b,c)$.

\begin{theorem}[Canonical Base Property]
Suppose $tp(a/K)$ is finite-dimensional and $L>K$ is an algebraically closed differential field (in $\U$). Then, the type
$$tp(Cb(a/L)/K\langle a\rangle)$$
is internal to the constants.
\end{theorem}
\begin{proof}
Now that we have suitable adaptations of $\D$-modules and D-varieties, the proof follows the same lines as the proof in the ordinary case DCF$_{0,1}$ (\cite[Theorem~1.1]{PiZi}). Nonetheless, for completeness and exposition sake, we provide details.

By Fact~\ref{Dvar}(2), we may assume that $a$ is the differential-generic point of $(V,s)^\#$ for some D-variety $(V,s)$ over $K$. Now let $W$ be the Kolchin-locus of $a$ over $L$. Then, $W$ is a D-subvariety of $V$ (i.e., $s(W)\subseteq \tau_\D W$) and $a$ is a differential-generic point of $(W,s_W:=s|_V)^\#$ over $L$. As $s$ is defined over $K$, the canonical base of $a$ over $L$ is $\D$-interdefinable over $K$ with the minimal field of definition of $W$, call it $F$. It then suffices to show that $tp(F/K\langle a\rangle)$ is internal to $C_\U$.

For each $\ell>0$, equip the jet spaces $j_\ell W_a$ and  $j_\ell V_a$ with their canonical $\D$-module structures (see paragraph after Fact~\ref{Dvar}). Furthermore, as $W$ is a D-subvariety of $V$, the canonical embedding of $j_\ell W_a$ into $j_\ell V_a$ is also an embedding of $\D$-modules (as this map is the dual of the natural surjection $\M_{V,a}/\M_{V,a}^{\ell+1}\to \M_{W,a}/\M_{W,a}^{\ell+1}$ which is a $\D$-module map). Let $b_\ell$ be a $C_\U$-basis for $j_\ell V_a^\flat$ which is also a $\U$-basis for $j_\ell V_a$ (this exists by Lemma~\ref{Cbasis}) and set $B:=\bigcup_{\ell>0} b_\ell$. We may choose the $b_\ell$'s such that 
$F\ind_{K\langle a\rangle} B$.

With respect to the basis $b_r$, we obtain a $\D$-module isomorphism $\phi$ from $j_\ell V_a$ to $(\U^{d_\ell},\D)$ for some $d_\ell\in \mathbb N$, where in the latter module $\D$ applies to each entry of a (column) vector. The image of $j_\ell W_a$ under $\phi$ is a $\D$-submodule $S_\ell$ of $(\U^{d_\ell},\D)$. Note that then $S_\ell^\flat\subseteq C_\U^{d_\ell}$. Let $e_r$ be a $C_\U$-basis of $S_\ell^\flat$ which is also a $\U$-basis for $S_\ell$ (note that each $S_\ell$ is $e_\ell$-definable). Set $E:=\bigcup_{\ell>0}e_\ell\subset C_\U$.

Now it is just a matter of checking that $F\subseteq \dcl(a,K,B,E)$. Let $\sigma$ be an automorphism of $(\U,\D)$ fixing $a,K,B,E$ pointwise. It suffices to show that $\sigma$ fixes $W$ setwise (as then it will fix $F$ pointwise, being the field of definition of $W$). As $j_\ell V_a$ is defined over $K(a)$, we have 
$$j_\ell \sigma(W)_a=j_\ell \sigma(W)_{\sigma(a)}=\sigma(j_\ell W_a)\subseteq \sigma(j_\ell V_a)=j_\ell V_a.$$
Furthermore, as $S_\ell$ is $E$-definable and the $\D$-module isomorphism $\phi$ is over $B$, we get $\sigma(j_\ell W_a)=j_\ell W_a$. Altogether we have shown that as subspaces of $j_\ell V_a$, we have $j_\ell W_a=j_\ell \sigma(W)_a$ for all $\ell>0$. Hence, by Fact~\ref{unique}, $W=\sigma(W)$ as claimed.
\end{proof}

Now a standard argument (see \cite[Corollary 6.19]{MooScan} or \cite[Corollary 3.10]{PiZi}, for instance) yields, from the CBP, the expected form of the dichotomy for finite-dimensional types.

\begin{corollary}[Dichotomy for finite-dimensional types]\label{Zilfin}
Let $p$ be a finite-dimensional type over $K$ of U-rank one. Then, $p$ is either locally modular or nonorthogonal to the constants.
\end{corollary}

\begin{remark}
As we mentioned in the introduction, this finite-dimensional form of the dichotomy already appears in unpublished work of Hrushovski and Sokolovi\'c~\cite{HruSok}. Their proof goes via Zariski geometries rather than deploying the canonical base property.
\end{remark}

\

\section{Proof of the dichotomy}\label{dichotomy}

We now prove the dichotomy for arbitrary minimal types (i.e., not necessarily finite-dimensional). The proof is based on the approach of Bustamante for the dichotomy in the differential-difference setting DCFA$_{0}$ \cite{Busta}. Namely, we deploy the analysis of regular types types in DCF$_{0,m}$ \cite{MooPiScan} to show that a nonlocally modular minimal type must be finite-dimensional (and now one can refer to Corollary~\ref{Zilfin}). As before, $(\U,\D)$ is a sufficiently saturated model of DCF$_{0,m}$ and $K$ is a small algebraically closed differential subfield.

\begin{theorem}
Let $p\in S(K)$ be of U-rank one. Then, $p$ is either locally modular or nonorthogonal to the constants. 
\end{theorem}
\begin{proof}
Assume $p$ is nonlocally modular. It suffices to prove that in this case $p$ is finite-dimensional, as then we can invoke Corollary~\ref{Zilfin}. As $p$ is a regular (by minimal rank) nonlocally modular type, by \cite[Theorem 3.17]{MooPiScan}, there is a definable (possibly over additional parameters) subgroup $G$ of the additive group $\mathbb G_a$ whose generic type $\mathfrak g_G$ is regular and nonorthogonal to $p$. Using Lascar inequalities, we see that if the Cantor normal form of $U(G)$ is
$$U(G)=\omega^{\beta_1}n_1+\dots +\omega^{\beta_k}n_k$$
with $\beta_1>\cdots>\beta_k\geq 0$ ordinals and the $n_i$'s positive integers, then $\beta_k=0$. That is, the Cantor form of $U(G)$ has a nonzero constant term. Indeed, towards a contradiction, assume $\beta_k\neq 0$. Since $p\not\perp \mathfrak g_G$, there is a set of parameters $A$ such that $a\nind_A b$ where $a$ and $b$ realise the nonforking extensions of $p$ and $\mathfrak g_G$ to $A$, respectively. On the one hand, Lascar inequalities says
$$U(a/Ab)+ U(b/A) \leq U(a,b/A)\leq U(a/Ab) \oplus U(b/A).$$
Since $p$ is minimal, $U(a/Ab)=0$, and so $U(a,b/A)=U(G)$. On the other hand, Lascar inequalities also yields
$$U(a,b/A) \leq U(b/Aa) \oplus U(a/A)= U(b/Aa)\oplus 1.$$
But $U(b/Aa)<U(G)$, and so, using that $\beta_k\neq 0$, we get $U(b/Aa) \oplus 1<U(G)$, a contradiction.

Now, by the Berline-Lascar decomposition theorem \cite[Theorem 6.7]{Poi}, there is a definable subgroup $H\leq G$ such that $U(G/H)=n_k$; in other words, $G/H$ has nonzero finite rank. As $G$ and $H$ are definable subgroups of the additive group, a result of Cassidy \cite[Proposition 11]{Cass} states that they are given as zero sets of linear homogeneous differential polynomials. Furthermore, if $f_1,\dots,f_n$ are such defining $H$, then the image of the map $(f_1,\dots,f_n):G\to \mathbb \U^{n}$ is a definable subgroup of $\mathbb G_a^n$ which is isomorphic to $G/H$; in other words, we may identify the quotient group $G/H$ with a definable subgroup of $\mathbb G_a^n$ and hence it is also defined by linear homogeneous differential equations. It then follows that $G/H$ is a $C_{\U}$-vector subspace of $\U^n$; from this we obtain that the generic type $\mathfrak g_{G/H}$ of $G/H$ must be \emph{finite-dimensional} (otherwise, $G/H$ would have infinite dimension as a $C_\U$-vector space, and any such space has infinite $U$-rank).

\medskip

\noindent {\bf Claim.} There is a finite-dimensional minimal type $q$ that is nonorthogonal to the generic type $\mathfrak g_{G/H}$ of $G/H$.

\begin{proof}[Proof of Claim]
Suppose $G$ and $H$ are defined over some algebraically closed differential field $L$ (note that then so is $G/H$). From the theory of coordinatisation in finite U-rank, see Lemma 5.1 of Chapter 2 in \cite{Pillay}, there is a (stationary) type $q$ with $U(q)=1$ such that $q\not\perp \mathfrak g_{G/H}$. In the proof of that lemma, $q$ is of the form $tp(c/E)$ for some $E>L$ where $c$ is interdefinable with $Cb(r)$ over $L$ and $r$ is a (forking) extension of $\mathfrak g_{G/H}$ with $U(r)=U(\mathfrak g_{G/H})-1$. Let $(a_i:i<\omega)$ be a Morley sequence in $r$, then $c\in Cb(r)\subseteq \dcl(a_i:i<\omega)$, see \cite[\S2, Lemma 2.28]{Pillay} for instance. As each $tp(a_i/E)$ is finite-dimensional (since $a_i\models \mathfrak g_{G/H})$, we obtain that $q=tp(c/E)$ is also finite-dimensional.
\end{proof}

As the quotient map $G\to G/H$ induces a definable map from $\mathfrak g_G$ to $\mathfrak g_{G/H}$, it follows that $\mathfrak g_G$ is also nonorthogonal to $q$. Now, to finish the proof, by transitivity of nonorthogonality for regular types (in this case $p$, $\mathfrak g_G$, and $q$), $p\not\perp q$. As $p$ is minimal, the finite-dimensionality of $q$ implies that $p$ is finite-dimensional as well. 
\end{proof}

\


\end{document}